\documentclass[a4paper,12pt,reqno]{amsart}
\usepackage[utf8]{inputenc}
\usepackage{amssymb}
\usepackage{amsmath}
\usepackage{amsthm}
\usepackage{amsfonts}
\usepackage{mathtools}
\usepackage{makecell}
\usepackage{multirow}
\usepackage{hyperref}
\usepackage{cleveref}
\usepackage{todonotes}
\usepackage{comment}
\usepackage{lipsum}
\usepackage{mathrsfs}
\usepackage{enumitem}
\usepackage[all,cmtip]{xy}
\makeatletter
\def\dicho#1{\expandafter\@dicho\csname c@#1\endcsname}
\def\@dicho#1{\ifnum#1>1 or \else\fi(\@Roman#1)}
\makeatother
\AddEnumerateCounter{\dicho}{\@dicho}{or (III)}
\newlist{dichotomy}{enumerate}{1}
\setlist[dichotomy]{label=\dicho*,leftmargin=1.5cm}

\usepackage[british]{babel}
\usepackage[margin=1.2in]{geometry}
\usepackage{eqparbox}

\usepackage{pbox}

\newcommand{\set}[1]{\left\lbrace #1 \right\rbrace}

\newcommand{\field}[1]{\mathbb{#1}}  % the font for a mathematical field is blackboard
\newcommand{\Q}{\field{Q}} % the field of the rationals
\newcommand{\R}{\field{R}} % the field of the reals
\newcommand{\Z}{\field{Z}} % the ring of integers
\newcommand{\F}{\field{F}} % finite fields
\renewcommand{\P}{\field{P}}

\newcommand{\Mod}[1]{\ (\mathrm{mod}\ #1)}

    \DeclareFontFamily{U}{wncy}{}
    \DeclareFontShape{U}{wncy}{m}{n}{<->wncyr10}{}
    \DeclareSymbolFont{mcy}{U}{wncy}{m}{n}
    \DeclareMathSymbol{\Sha}{\mathord}{mcy}{"58}

\DeclareMathOperator{\rk}{rk}

\DeclareMathOperator{\Aut}{Aut}

\DeclareMathOperator{\Sel}{Sel}

\newtheorem{lemma}{Lemma}
\newtheorem{theorem}[lemma]{Theorem}
\newtheorem{proposition}[lemma]{Proposition}

\newtheorem{corollary}[lemma]{Corollary}
\newtheorem{conjecture}[lemma]{Conjecture}

\theoremstyle{definition}

\newtheorem{remark}[lemma]{Remark}

\newtheorem*{ack}{Acknowledgements}

\numberwithin{lemma}{section}
\numberwithin{equation}{section} 
\numberwithin{figure}{section}

\title[Torsion subgroups of elliptic curves and Granville's conjecture]{Torsion subgroups of elliptic curves over quadratic fields and a conjecture of Granville}
\author{Barinder S. Banwait}
\address{Barinder S. Banwait \\
Department of Mathematics \& Statistics\\  
Boston University\\
Boston, MA\\
USA}
\email{barinder.s.banwait@gmail.com}

\author{Maarten Derickx}
\address{Maarten Derickx,
Den Haag,
The Netherlands}
\email{maarten@mderickx.nl}
\date{}

\makeatletter
\providecommand\@dotsep{5}
\renewcommand{\listoftodos}[1][\@todonotes@todolistname]{%
  \@starttoc{tdo}{#1}}
\makeatother


\begin{document}

\maketitle

\begin{abstract}
We study the problem of determining the groups that can arise as the torsion subgroup of an elliptic curve over a fixed quadratic field, building on work of Kamienny-Najman, Krumm, and Trbović. By employing techniques to study rational points on curves developed by Bruin and Stoll, we determine the possible torsion subgroups of elliptic curves over quadratic fields $\mathbb{Q}(\sqrt{d})$ for all squarefree $d$ with $|d| < 800$, improving on the previously known range of $-5 < d < 26$. We use our computations to study the validity of a conjecture of Granville concerning how many twists of a given hyperelliptic curve admit a nontrivial rational point.
\end{abstract}

% \listoftodos

\section{Introduction}

Let $E$ be an elliptic curve over a number field $K$. The Mordell-Weil theorem establishes that the group $E(K)$ of $K$-rational points of $E$ is finitely generated as an abelian group, so one has an isomorphism of groups \[E(K) \cong E(K)_{tors} \oplus \Z^r,\] where $E(K)_{tors}$ is a finite abelian group called the \textbf{torsion subgroup} of $E/K$, and $r$ is called the \textbf{rank} of $E/K$.

The question of classifying which possible torsion subgroups may arise as one varies over elliptic curves over a fixed number field $K$ - what in this paper we refer to as \textbf{uniformity of torsion over $K$} - goes back to Levi's 1908 ICM address in Rome \cite{levi1909sull}, in which he conjectured that for $K = \Q$, there are only $15$ possible groups that can arise. As is well-known, this conjecture was finally established in Mazur's seminal work \cite{Mazur3}; see \cite{schappacher1996beppo} for an interesting history of this landmark result in the arithmetic of elliptic curves.

\begin{theorem}[Mazur (1977)]\label{thm:mazur}
    Let $E$ be an elliptic curve over $\Q$. Then $E(\Q)_{tors}$ is isomorphic to one of the following $15$ groups:
\[
\begin{aligned}
\Z/N\Z &\hspace{3em} 1 \leq N \leq 10 \mbox{ or } N = 12;\\
\Z/2\Z \oplus \Z/2N\Z &\hspace{3em} 1 \leq N \leq 4.
\end{aligned}
\]
\end{theorem}

Since Mazur's work in the 70s, much progress has been made towards a closely related question that in the literature is referred to as \textbf{strong uniformity of torsion}. Here, instead of fixing a number field $K$, one fixes an integer $d$, and asks for a classification $\Phi(d)$ of possible torsion subgroups of \emph{any} elliptic curve over \emph{any} number field of degree $d$ over $
\Q$. Mazur's result above establishes this for $d = 1$. Subsequently, in a long series of papers in the 80s by Kenku, Momose, and Kamienny \cite{Kamienny:PointsOrderp,Kamienny:TorsionSubgroupTotallyReal,Kamienny:TorsionAllQuadratic,Kamienny:TorsionAllQuadratic2,Kamienny:TorsionPoints,Kamienny:TorsionQuadratic,Kenku:TwoPowerTorsionQuadratic,Kenku:CertainTorsionQuadratic,Kenku:ModularCurvesQuadratic,Kenku:RationalTorsionQuadratic,KenkuMomose:TorsionQuadratic,momose1984p,Momose:Xsplitp}, the $d = 2$ case was studied; the culmination of the work of Kenku and Momose was \cite{KenkuMomose:TorsionQuadratic}, in which they proposed a list of $26$ possible torsion subgroups for elliptic curves over quadratic fields. The authors showed that these $26$ groups arise infinitely often as one varies both the elliptic curve and the quadratic field, and moreover showed that this list would be complete if one can show that no elliptic curve over a quadratic field admits a torsion point of prime order $p > 13$. This fact about bounding the so-called \textbf{torsion primes} $S(d)$ in degree $d = 2$ by $13$ was subsequently established by Kamienny \cite{Kamienny:TorsionQuadratic}; one therefore has the following result that we refer to as the \textbf{Kamienny-Kenku-Momose} (hereafter KKM) classification.

\begin{theorem}[Kamienny-Kenku-Momose (1992)]\label{thm:strong_quadratic_uniformity}
Let $E$ be an elliptic curve over a quadratic field $K$. Then $E(K)_{tors}$ is isomorphic to one of the following $26$ groups:
\[
\begin{aligned}
\Z/N&\Z \hspace{3em} 1 \leq N \leq 16 \mbox{ or } N = 18;\\
\Z/2\Z \oplus \Z/2N&\Z \hspace{3em} 1 \leq N \leq 6;\\
\Z/3\Z \oplus \Z/3N&\Z \hspace{3em} 1 \leq N \leq 2;\\
\Z/4\Z \oplus \Z/4&\Z.
\end{aligned}
\]
\end{theorem}

Kamienny and Mazur subsequently showed \cite{kamienny1995rational} that finiteness of $S(d)$ is equivalent to finiteness of $\Phi(d)$, albeit in a non-effective way (that is, knowing exactly what $S(d)$ is does not immediately allow one to determine $\Phi(d)$), and Merel established finiteness of $S(d)$ for all $d$ \cite{merel1996bornes}, proving the erstwhile Uniform Boundedness conjecture for torsion primes of elliptic curves over number fields. The only other value of $d$ for which $\Phi(d)$ is known fully is $d = 3$, a recent result due to the second author with Etropolski, van Hoeij, Morrow, and Zureick-Brown \cite{derickx2021sporadic}. There are partial results known for $d = 4, 5, 6$ and $7$; see e.g. \cite{derickx2017torsion} or \cite{sutherland2012torsion} and the references contained in the introduction there for the state-of-the-art known about strong uniformity in higher degree number fields.

In this paper, however, we will return to the original uniformity question, and attempt to classify the torsion subgroups of elliptic curves over fixed quadratic fields $K$.

Najman was the first to consider this question, determining which of the $26$ groups in the KKM classification actually arise over each of the cyclotomic quadratic fields $\Q(\zeta_3)$ and $\Q(\zeta_4)$ \cite{najman2011torsion}. Very soon thereafter, Kamienny and Najman \cite{kamienny2012torsion} determine the smallest quadratic field (ordered by absolute value of discriminant) realising each of the 26 torsion subgroups in the KKM classification, and obtain results about the rank of elliptic curves having prescribed torsion. This latter phenomenon of the interplay between rank and torsion was further investigated by Bosman, Bruin, Dujella and Najman \cite{bosman2014ranks}, and yielded very striking results such as any elliptic curve over any quadratic field with torsion subgroup $\Z/13\Z$ must have even rank.

The idea of determining the torsion subgroups of elliptic curves over all quadratic fields $\Q(\sqrt{d})$ in some discriminant range first appears in a paper of Trbović \cite{trbovic2020torsion}, who attempted such a classification for $0 < d < 100$; there were however $16$ values of $d$ in this range for which she was unable to fully decide on which torsion subgroups arise; more precisely, for these $16$ values, it remained undecided whether $\Z/16\Z$ arises as a possible torsion subgroup. Because of these $16$ unknowns, a result classifying torsion in a range of $d$ was only established for $0 < d < 24$. Coupling with this Najman's work \cite{najman2011torsion}, the range of known quadratic torsion is currently $-5 < d < 24$.

The main result of this paper resolves the situation for these $16$ values, considers negative values of $d$, and significantly extends the discriminant range. To succinctly state our results, we make the following observations:

\begin{itemize}
    \item the $15$ groups in Mazur's classification (which form a subset of the $26$ groups in the KKM classification) arise over every quadratic field; see \cite[Section 2]{trbovic2020torsion}.
    \item the groups $\Z/3\Z \oplus \Z/3\Z$ and $\Z/3\Z \oplus \Z/6\Z$ (respectively, the group $\Z/4\Z \oplus \Z/4\Z$) correspond to having full $3$ (respectively, $4$) torsion, and so by a standard corollary of Galois equivariance of the Weil pairing, can arise only over $\Q(\zeta_3) = \Q(\sqrt{-3})$ (respectively, $\Q(\zeta_4) = \Q(\sqrt{-1})$). 
\end{itemize}

This reduces the problem to determining, for each quadratic field in our range, which of the remaining $8$ torsion subgroups arise over that quadratic field. For a group $\Z/M\Z \oplus \Z/N\Z$ (with $M | N$) and a positive integer $B$, we therefore make the following definition: \[ T_B(M,N) := \left\{|d| < B \mbox{ squarefree }: \Z/M\Z \oplus \Z/N\Z \mbox{ is a torsion group over } \Q(\sqrt{d})\right\}. \] As usual\footnote{the same convention is used for modular curves}, we write $T_B(N) := T_B(1,N)$. Our task is then to determine, for some chosen value of $B$, the $8$ sets
\begin{align}\label{eqn:the_8_cases}
\begin{split}
    \mbox{ genus 1 } &: \ T_B(11), T_B(14), T_B(15), T_B(2,10), T_B(2,12)\\
    \mbox{ genus 2 } &: \ T_B(13), T_B(16), T_B(18).
\end{split}
\end{align}
These sets have been labelled with a genus that corresponds to the genus of the modular curve $X_1(M,N)$ that plays the role of a moduli space of elliptic curves having $\Z/M\Z \oplus \Z/N\Z$ as a subgroup of their torsion subgroup, and the genus plays a significant role in the arithmetic geometry of these modular curves. The genus $2$ cases pose a more significant challenge, and most of the paper will be focussed on these three cases. Indeed, Kamienny and Najman \cite{kamienny2012torsion} already showed that determination of the genus $1$ sets above correspond to computing whether the rank of the $d$-twist of each of the associated (elliptic) modular curves is positive; we carry this out for $B = 10{,}000$ in \Cref{sec:genus-1}.

We may now state our main result.

\begin{theorem}\label{thm:main}
    \begin{enumerate}
        \item Letting $S_{13}$ denote the set 
        \begin{align*}
    S_{13} := &\left\{ 17, 113, 193, 313, 481, 1153, 1417, \right. \\ &\ \left. 2257, 3769, 3961, 5449, 6217, 6641, 9881 \right\},
    \end{align*}
        we have \[ S_{13} \subseteq T_{10{,}000}(13) \subseteq S_{13} \cup \left\{9689\right\}.\]
        \item Letting $S_{18}$ denote the set \[ S_{18} = \left\{ 33, 337, 457, 1009, 1993, 2833, 7369, 8241, 9049 \right\}, \] we have \[ S_{18} \subseteq T_{10{,}000}(18) \subseteq S_{18} \cup \left\{2841, 4729, 9969\right\}.\]
        
    \item We have \begin{align*}
    T_{800}(16) = &\left\{-671, -455, -290, -119, -15, 10, 15, 41, 51, \right. \\
    &\ \left. 70, 93, 105, 205, 217, 391, 546, 609, 679 \right\}.
    \end{align*}
    \end{enumerate}
\end{theorem}

Together with the aforementioned computation of ranks of twists of the five elliptic modular curves alluded to above, this result gives a resolution of the uniformity of torsion question for every quadratic field $\Q(\sqrt{d})$ for $|d| < 800$, offering a significant improvement over the previously known range $-5 < d < 24$. Moreover, for $13$ and $18$-torsion, only the sets $T_{256}(13)$ and $T_{680}(18)$ were previously known \cite[Theorems 2.7.7 and 2.7.8]{krummthesis}; this illustrates that parts (1) and (2) above also greatly improve upon previous results.

We briefly describe our methods, based on the aforementioned work of Najman and Trbovi\'{c}, and give an overview of the paper. The problem may be expressed as determining, for each quadratic field $\Q(\sqrt{d})$ in our range, which of the $8$ modular curves referred to in \eqref{eqn:the_8_cases} admit a noncuspidal $\Q(\sqrt{d})$-point. As mentioned above, for the five genus $1$ modular curves, \Cref{sec:genus-1} explains how this essentially boils down to computing the rank of the $d$-twist of each (elliptic) modular curve, something that we carry out in Magma \cite{magma}. The torsion groups $\Z/13\Z$ and $\Z/18\Z$ are dealt with in \Cref{sec:13_18}, where we use work of Krumm \cite{krummthesis} that reduces the problem to determining the existence of a $\Q$-rational point on the $d$-twist of the modular curves $X_1(13)$ and $X_1(18)$. Krumm already used \emph{Two-cover descent} to resolve this problem for many values of $d$; here we augment this method with two improvements: a necessary condition on the rank of the Jacobian varieties $J_1^d(13)$ and $J_1^d(18)$ of the twisted modular curves, and an application of the \emph{Mordell-Weil sieve}. This yields parts (1) and (2) of \Cref{thm:main}. The reason for the slight ambiguity concerning the values $9689$, $2841$, $4729$ and $9969$ is that the Mordell-Weil sieve method failed here, as we were unable to find explicit generators of the Mordell-Weil group of the above Jacobians.

Dealing with $\Z/16\Z$ as a possible torsion subgroup is the most difficult, since every twist admits a $\Q$-rational point, so we are essentially forced to compute all $\Q$-rational points, and this is the reason for the significantly smaller value of $B = 800$ in part (3). In this range, the \emph{Elliptic curve Chabauty} method is successful in doing this, establishing Part (3) of \Cref{thm:main} in \Cref{sec:16}. 

Resolving the genus $2$ cases comes down to determining whether quadratic twists of these curves in a discriminant range admit a nontrivial rational point. This is something that Granville has previously studied \cite{granville2007rational}; in particular, a conjecture about how many such twists should admit nontrivial points was given there. We compare our results with this conjecture; there is an apparent discrepancy between our data and his conjecture, which we explore in \Cref{sec:results}. Finally in \Cref{sec:future} we indicate some avenues for future research.

We have mainly used Magma in our computations, although some parts have been done in SageMath \cite{Sage}. Our code implementations may be found at
\begin{center}
\url{https://github.com/isogeny-primes/quadratic-torsion}
\end{center}
All filenames given in the paper will refer to files in this repository.

\ack{
We are grateful to Jennifer Balakrishnan, Peter Bruin, Aashraya Jha, Steffen M\"{u}ller, Filip Najman and Alexander Smith for helpful comments and correspondence. We are grateful to Michael Stoll for answering queries about getting his Magma implementation of the Mordell-Weil sieve to work in our setting.

The first named author is supported by Simons Collaboration grant ID \#550023 for the Collaboration on Arithmetic Geometry, Number Theory, and Computation.
}

\section{The five genus 1 cases}\label{sec:genus-1}

In this section we deal with the genus $1$ cases of \eqref{eqn:the_8_cases}; that is, for every quadratic field $\Q(\sqrt{d})$ with $|d| < 10{,}000$, we determine if each of the five modular curves of genus one, viz. $X_1(11)$, $X_1(14)$, $X_1(15)$, $X_1(2, 10)$ and $X_1(2, 12)$, admits a noncuspidal $\Q(\sqrt{d})$-rational point. The following result of Kamienny and Najman shows that this essentially comes down to determining whether or not the rank of the elliptic modular curve over $\Q(\sqrt{d})$ is positive.

\begin{theorem}[Kamienny-Najman, Theorems 15 and 16 in \cite{kamienny2012torsion}]
\
\begin{enumerate}
    \item If $X_1(11)$, $X_1(2, 10)$ or $X_1(2, 12)$ possess a noncuspidal quadratic point, then that point has infinite order.
    \item $X_1(14)$ possesses a noncuspidal $\Q(\sqrt{d})$-torsion point of finite order if and only if $d = -7$.
    \item $X_1(15)$ possesses a noncuspidal $\Q(\sqrt{d})$-torsion point of finite order if and only if $d = -15$.
\end{enumerate}
\end{theorem}

Since, for $E$ an elliptic curve over $\Q$, one has \[ \rk(E(\Q(\sqrt{d})) = \rk(E(\Q)) + \rk(E^d(\Q)), \] where $E^d$ denotes the $d$\textsuperscript{th} quadratic twist of $E$, dealing with these five cases amounts to checking the nonzeroness or otherwise of the $\Q$-rank of several quadratic twists of these elliptic modular curves (noting that the $\Q$-rank of the five original curves is zero). Of course, for $d = -7$ and $-15$, one only has the three curves in (1) to deal with.

Hereafter, $E$ will denote one of the above five elliptic modular curves. We first check via a modular symbols calculation in Sage involving the \emph{twisted winding element} (see \cite[Section 2.2.2]{bosmanthesis} or \cite[Section 6.3.3]{couveignes2011computational} for more details) whether or not the \emph{analytic rank} of $E^d$ is zero. Here we observe that $X_1(2,10)$ is isogenous to $X_0(20)$, and $X_1(2,12)$ is isogenous to $X_0(24)$ (as may be verified in the LMFDB \cite{lmfdb}). The main routine in \path{quadratic_torsion/positive_rank_twists.py} produces a list of values of $d$ for which the analytic rank is positive; these values have been stored in the \path{genus_one_lists} folder. 

If the analytic rank is zero, then by Kolyvagin \cite{kolyvagin1989finiteness}, we know that the rank is zero. If it is nonzero, we then verify in \path{magma_scripts/elliptic.m} that the algebraic rank is nonzero by the standard descent method implemented in Magma.

\section{\texorpdfstring{$X_1(13)$ and $X_1(18)$}{X1(13) or X1(18}}\label{sec:13_18}

In this section we prove parts (2) and (3) of \Cref{thm:main}.  We use the models for these two modular curves from \cite[Section 2]{kamienny2012torsion}:
\begin{align*}
    X_1(13) : y^2 &= x^6 - 2x^5 + x^4 - 2x^3 + 6x^2 - 4x + 1;\\
    X_1(18) : y^2 &= x^6 + 2x^5 + 5x^4 + 10x^3 + 10x^2 + 4x + 1.
\end{align*}
Our task is to decide, for every squarefree $d$ with $|d| < 10{,}000$, whether each of these modular curves admits a noncuspidal $\Q(\sqrt{d})$-rational point. For this we start by employing the methods outlined by Krumm in Section 2 of his thesis \cite{krummthesis}. However we add two important improvements.

The first of these is \Cref{prop:pos-rank-13-18}, which implies for $N=13$ or $18$ that if $J_1^d(N)$ has rank zero, that there are no such $\Q(\sqrt{d})$-rational points. In particular, the use of the magma function \verb|Chabauty0| to compute all point on $X_1^d(N)(\Q)$ as mentioned in \cite[\S 2.5.1]{krummthesis} is no longer necessary since we prove that this set is always empty if $d\neq -3$ and  $J_1^d(N)$ has rank zero. The elimination of this computational step, and using modular symbol computations to determine the analytic rank of $J_1^d(N)$ explains why we could extend our computation to a much large range of $d$.

The second improvement is that we apply the Mordell-Weil sieve to try and show $X_1^d(N)(\Q)=\emptyset$ in cases where the methods of Krumm fail. This extra step explains why we can show for example that 18-torsion does not occur over $\Q(\sqrt{681})$, while this is one of the cases that Krumm couldn't handle.

Throughout this section, we shall use $N$ to denote either $13$ or $18$. We may also ignore the values $d = -1$ and $-3$ since, as mentioned in the introduction, Najman has already dealt with these.

Krumm shows that any noncuspidal $\Q(\sqrt{d})$-rational point on $X_1(N)$ must have $\Q$-rational $x$-coordinate, and therefore yields a $\Q$-rational point on the $d$-twisted modular curve $X_1^d(N)$. Conversely, if $X_1^d(N)$ admits a $\Q$-rational point, then this in turn would correspond to a noncuspidal $\Q(\sqrt{d})$-rational point on $X_1(N)$; these facts are proved in \cite[Lemma 2.7.3]{krummthesis}. This reduces the task to checking whether any of these twists of these two modular curves possess a $\Q$-rational point.

If such a twist of $X_1(N)$ possesses a rational point, then several necessary conditions must be satisfied. We collect these into the following subsections.

\subsection{Everywhere Local solubility}

If a curve $X$ over a number field $K$ admits a $K$-rational point, then it certainly admits a point rational over every completion of $K$; i.e., the curve must be everywhere locally soluble. This is a finite computation to check for any explicitly given curve, for which Magma has an implementation.

\subsection{Congruence conditions on \texorpdfstring{$d$}{d}}

Krumm provides a necessary condition on $d$ for there to exist a $\Q$-point on $X_1^d(N)$.

\begin{proposition}[Krumm, Theorems 2.6.5 and 2.6.9 in \cite{krummthesis}]
\ 
    \begin{enumerate}
        \item If $X_1^d(13)(\Q) \neq \emptyset$, then:
        \begin{enumerate}
            \item $d > 0$;
            \item $d \equiv 1 \Mod{8}$.
        \end{enumerate}
        \medskip
        \item If $X_1^d(18)(\Q) \neq \emptyset$, then:
        \begin{enumerate}
            \item $d > 0$;
            \item $d \equiv 1 \mbox{ or } 9 \Mod{24}$.
        \end{enumerate}
    \end{enumerate}
\end{proposition}

\begin{remark}
    That $d$ must be positive here was independently proved by Bosman, Bruin, Dujella, and Najman; see \cite[Theorem 9]{bosman2014ranks}.
\end{remark}

\subsection{Two-cover descent}

This is a technique due to Bruin and Stoll \cite{bruin2009two}. Briefly, they provide a refinement of a classical theorem of Chevalley and Weil \cite{chevalley1932un} to say that, for any fixed $n \geq 2$, if a hyperelliptic curve $C$ over a number field $k$ admits a $k$-rational point, then this rational point must have a rational preimage on one of finitely many covering curves of a particular form depending on $n$, called $n$-covers. By considering the set of (isomorphism classes of) everywhere locally soluble $n$-coverings, called the $n$-Selmer set of $C$, one has the result that if this Selmer set is empty, then $C$ has no $k$-rational points. Bruin and Stoll make this explicit and algorithmic in the case $n = 2$, by working with a closely related Selmer object (the \emph{fake $2$-Selmer group}), and most importantly, provide a Magma implementation of this for genus $2$ curves over $\Q$, accessible via the intrinsic \path{TwoCoverDescent}. The upshot is that another necessary condition for $X_1^d(N)(\Q)$ to be nonempty is that its fake $2$-Selmer set must be nonempty.

\subsection{Positive rank of \texorpdfstring{$J_1^d(N)$}{J1d(N)}}

One necessary condition that we introduce is \Cref{cor:pos-rank-13-18} below, which is that the rank of $J_1^d(N)(\Q)$ must be positive. We first establish the following preparatory lemmas concerning torsion growth in $J_1(N)$ over quadratic fields.

\begin{lemma}\label{lem:no_torsion_growth_13}
    For every quadratic field $K$, we have \[ J_1(13)(K)_{tors} = J_1(13)(\Q)_{tors} \cong \Z/19\Z. \]
\end{lemma}

\begin{proof}
    For $p \geq 5$, $p \neq 13$, the torsion subgroup $J_1(13)(K)_{tors}$ injects into the reduction $\widetilde{J_1(13)}(\F_{p^2})$ (see \cite[Appendix]{katz1980galois}). By computing this latter group for $p = 5$ and $17$, one sees that it must be a subgroup of $\Z/19\Z$. On the other hand, the torsion over $\Q$ is $\Z/19\Z$, as may be directly verified in Magma. These computations may be found in \path{magma_scripts/torsionVerifications.m} (this also includes the verifications for \Cref{lem:no_torsion_growth_18} below).
\end{proof}

\begin{remark}
    That $J_1(13)(\Q)_{tors} \cong \Z/19\Z$ was first proved by Ogg in \cite{ogg1972rational}, and this discovery is of great historical importance in the arithmetic of elliptic curves. Shortly after finding a point of order $19$ on $J_1(13)$, Ogg passed through Cambridge, MA, and communicated this discovery to Mazur and Tate; this inspired them to prove that in fact $J_1(13)(\Q)$ consists only of the $19$ torsion points; i.e., that it has zero rank over $\Q$; this has the corollary that no elliptic curve over $\Q$ possesses a rational point of order $13$ \cite{mazur1973points}. This argument was shortly thereafter generalised by Mazur to deal with all primes $p \geq 13$; combined with existing work of Kubert \cite{Kubert:Torsion}, this provided the classification of torsion subgroups of rational elliptic curves (\Cref{thm:mazur}) mentioned at the beginning of this paper.
\end{remark}

\begin{lemma}\label{lem:no_torsion_growth_18}
    For $K$ any quadratic field that is not $\Q(\sqrt{-3})$, we have \[ J_1(18)(K)_{tors} = J_1(18)(\Q)_{tors} \cong \Z/21\Z. \] Furthermore,
    \[ J_1(18)(\Q(\sqrt{-3}))_{tors}  \cong \Z/3\Z \oplus \Z/21\Z. \]
\end{lemma}

\begin{proof}
    The last part was already proved by Najman \cite[Lemma 7]{najman2010complete}. For the first part, by computing the group structure of $\widetilde{J_1(18)}(\F_{p^2})$ for prime p in the range $5 \leq p \leq 30$, we obtain that, for all quadratic fields $K$, we have $J_1(18)(K(\zeta_3))_{tors} = J_1(18)(\Q(\zeta_3))_{tors}$ (noting that the residue fields of $K(\zeta_3)$ are always contained in $\F_{p^2}$ for $p$ as above), and that $J_1(18)(K)_{tors}$ is at most $\Z/3\Z \oplus \Z/21\Z$. That the torsion over $\Q$ is isomorphic to $\Z/21\Z$ is a straightforward Magma computation. That the $3$-torsion rank of $J_1(18)$ is the same over $K(\zeta_3)$ as over $\Q(\zeta_3)$ implies that there can be no extra torsion attained over $K$.
\end{proof}

\begin{proposition}\label{prop:pos-rank-13-18}
    Let $d \neq -3$ be a squarefree integer, $N \in \left\{13, 18\right\}$, and $K = \Q(\sqrt{d})$. If $X_1(N)(K) \neq X_1(N)(\Q)$, then $J_1(N)(K)$ and hence $J_1^d(N)(\Q)$ has positive rank.
\end{proposition}

\begin{proof}
    If $P$ is a $K$-point of $X_1(N)$ that is not a $\Q$-point, then it embeds under the Abel-Jacobi map to a $K$-point of $J_1(N)$ that is not a $\Q$-point. Therefore by \Cref{lem:no_torsion_growth_13,lem:no_torsion_growth_18} it must be of infinite order. The final assertion comes from the equality $\rk(J_1(N)(K)) = \rk(J_1(N)(\Q)) + \rk(J_1^d(N)(\Q))$.
\end{proof}

\begin{corollary}\label{cor:pos-rank-13-18}
    For $N \in \left\{13,18\right\}$ and $d \neq -3$, if $X_1^d(N)(\Q) \neq \emptyset$, then $J_1^d(N)(\Q)$ has positive rank.
\end{corollary}

\begin{proof}
    As explained at the beginning of this section (and shown by Krumm), rational points on $X_1^d(N)(\Q)$ correspond to noncuspidal $\Q(\sqrt{d})$-points on $X_1(N)$ that are not $\Q$-points; the result then follows from the previous proposition.
\end{proof}

As in \Cref{sec:genus-1}, we use the twisted winding element method to check whether or not the analytic rank of $J_1^d(N)(\Q)$ is positive; if it is zero, then by Kato's generalisation of the work of Kolyvagin-Logach\"{e}v \cite{kato2004p}, we know that the Mordell-Weil rank is zero.

Putting these four necessary conditions together - which is done in \path{computations/x1_13.m} and \path{computations/x1_18.m} (see also the log files there for the output) - we obtain
\begin{align*}
    T_{10,000}(13) \subseteq &\left\{17, 113, 193, 313, 481, 673, 1153, 1417, 1609, 1921, 2089, 2161, \right. \\
    &\ \left. 2257, 3769, 3961, 5449, 6217, 6641, 8473, 8641, 9689, 9881 \right\};\\
    T_{10,000}(18) \subseteq &\left\{33, 337, 457, 681, 1009, 1329, 1761, 1993, 2833, 2841, 2913, 3769, \right. \\
    &\ \left. 4729, 5281, 6217, 7057, 7321, 7369, 8241, 9049, 9969 \right\}.
\end{align*}
Out of these possible values of $d$, we run a search for rational points on $X_1^d(N)$; the values for which this search yields no rational points (and hence we expect that there are none) are as follows:
\begin{align*}
    \Z/13\Z &: 673, 1609, 1921, 2089, 2161, 8473, 8641, 9689 \\
    \Z/18\Z &: 681, 1329, 1761, 2841, 2913, 3769, 4729, 5281, 6217, 7057, 7321, 9969.
\end{align*}

To prove that these twists of $X_1(N)$ have no rational points, we employ the Mordell-Weil sieve, a technique also due to Bruin and Stoll \cite{bruin2010mordell} who have also provided a Magma implementation \cite{bruinstollsieve} for genus $2$ curves over $\Q$.

One input that one needs in order to use the Mordell-Weil sieving on a curve $C$ is a divisor of degree $3$ on $C$. Such a divisor is not always guaranteed to exist. However in the cases where we will apply it, we know \emph{a priori} that such a degree $3$ divisor has to exists by the following lemma.

\begin{lemma}
Let $C$ be hyperelliptic curve over a number field $K$ with an automorphism $\gamma$ of order $3$ such that $C/\langle \gamma \rangle$ has genus 0. Then every hyperelliptic quadratic twist $C'$ of $C$ that is everywhere locally solvable has a divisor of degree 3.
\end{lemma}

\begin{proof}
Since the hyperelliptic involution is unique, the automorphism of order 3 commutes with it. In particular $\gamma$ is also a $K$-rational isomorphism on every hyperelliptic quadratic twist. Now $C'/\langle \gamma \rangle$ is of genus 0 since over $\overline K$ it is isomorphic to $C/\langle \gamma \rangle$. The curve $C'$ is everywhere locally solvable, so $C'/\langle \gamma \rangle$ is everywhere locally solvable as well. By the Hasse principle for genus 0 curves there is a $K$ rational point $P$ in $C'/\langle \gamma \rangle(K)$. The pullback of $P$ along the quotient map $C' \to C'/\langle \gamma \rangle$ will then be a divisor of degree 3 on $C'$.
\end{proof}

The above proof also gives a practical algorithm to find this degree 3 point. Namely just search for rational points on the genus 0 curve $C'/\langle \gamma \rangle$ and pull them back to $C'$. Under the hypothesis of the lemma this curve is isomorphic to $\P^1_K$ so rational points will be easy to find.

\begin{remark}
We found that Mordell-Weil sieving in order to show that $X_1^d(N)$ is empty was much faster then we initially expected. The Mordell-Weil sieve works by choosing a suitable set $S$ of primes and an auxilary integer $N'$. The integer $N'$ here is actually called $N$ in \cite{bruin2010mordell}. The Mordell-Weil sieve applied to $X_1^d(N)$ works by trying to prove $X_1^d(N)= \emptyset$ using the commutativity of the the diagram 
\[
\centerline{\xymatrix{
 X_1^d(N)(\Q) \ar[r]\ar[d] & J_1^d(N)(\Q)/N'J_1^d(N)(\Q) \ar[d] \\
 \prod_{p\in S} X_1^d(N)(\F_p) \ar[r] & \prod_{p \in S} J_1^d(N)(\F_p)/N'J_1^d(N)(\F_p) 
}}
\]
and trying to show that $J_1^d(N)(\Q)/N'J_1^d(N)(\Q)$ and $\prod_{p\in S} X_1^d(N)(\F_p)$ have empty intersection in $\prod_{p \in S} J_1^d(N)(\F_p)/N'J_1^d(N)(\F_p)$. In order for this strategy to be successful we need the different $|J_1^d(N)(\F_p)|$ to share many common factors \cite[\S3.1]{bruin2010mordell} for the primes $p\in S$. It turns out that many of these common factors will also be factors of $N'$; the reason for this is that there is no new information learned from reducing mod $p$ if $J_1^d(N)(\F_p)/N'J_1^d(N)(\F_p) = \set{0}$. In running the Mordell-Weil sieve for $N=13$ we found that the value of $N'$ chosen by the Mordell-Weil sieve implementation of Bruin and Stoll \cite{bruinstollsieve} was often either $19$ or divisible by $19$. Here we give a heuristic explanation of why this is to be expected. 

Let $N'$ be some fixed integer. If the integers $|J_1^d(N)(F_{p})|$ are roughly uniformly distributed modulo $N'$, then the chance of $|J_1^d(N)(F_{p})|$ being divisible by $N'$ is roughly $1/N'$, which seems a reasonable assumption at first. However for $N=13$ and a fixed $d$ we have that $|J_1^d(13)(\F_p)|$ is far from randomly distributed modulo $19$. Indeed let $p$ be a prime that splits in $\Q(\sqrt{d})$; then $|J_1(13)(\F_p)|=|J_1^d(13)(\F_p)|$; however the left hand side is $0 \mod 19$ since $J_1(13)(\Q)$ contains a point of order $19$, meaning that $19 \mid |J_1^d(13)(\F_p)|$ for a density of at least $1/2$ of the primes. So the fact that a multiple of $19$ is often chosen in the Mordell-Weil sieve can be explained by this unusually high density of primes for which $19 \mid |J_1^d(13)(\F_p)|$. A similar story holds for $N=18$ where $N'$ was often divisible by $21$.
\end{remark}

Carrying out this strategy in \path{magma_scripts/MWSieve-x1_13.m} (and the analogous file for $X_1(18)$) dealt with all of the above values, \emph{except} $9689$ for $13$-torsion, and the three values $2841, 4729, 9969$ for $18$-torsion. In these cases, the reason for the failure was the call to \path{MordellWeilGroupGenus2}; this did not finish given the search bounds declared there, so we were unable to find explicit generators for the Mordell-Weil group of the Jacobian of the twist. It is possible that increasing these bounds and waiting longer might yield these generators in these cases. However, already for the value $8641$, it took over two days of Magma computation running on MIT's \path{lovelace} machine to furnish the generators.

This concludes the proof of parts (2) and (3) of \Cref{thm:main}.

\section{\texorpdfstring{$X_1(16)$}{X1(16)}}\label{sec:16}

In this section we prove part (4) of \Cref{thm:main}. We work with the following model, as before to be found in \cite[Section 2]{kamienny2012torsion}: \[ X_1(16) : y^2 = x(x^2 + 1)(x^2 + 2x - 1). \] On this model, the $x$-coordinates of the cusps are given by the following equation: \[ x(x - 1)(x + 1)(x^2 - 2x - 1)(x^2 + 2x - 1) = 0. \] As for $X_1(13)$ and $X_1(18)$, Krumm showed that any noncuspidal quadratic point must have rational $x$-coordinate on this model, and so corresponds to a $\Q$-rational point on the $d$-twist $X_1^d(16)$. However, unlike before, it is not the case that every $\Q$-rational point on $X_1^d(16)$ corresponds to a noncuspidal quadratic point on $X_1(16)$, because the point $(0,0)$ on $X_1(16)$ (which is a cusp) gives the rational point $(0,0)$ on every twist $X_1^d(16)$. In this way we see that our problem reduces to determining the existence or otherwise of a rational point on $X_1^d(16)$ with nonzero $y$-coordinate.

In particular, since $X_1^d(16)$ admits a rational point for every $d$, this is a different problem than that of the previous section. Indeed the existence of a global rational point prevents all local techniques for yielding any results. And hence we are essentially forced to compute all $\Q$-rational points on $X_1^d(16)$, rather than merely determining the existence of them, and it is this that makes this case the most difficult. Many of the necessary conditions in the previous section no longer apply. One that does survive, however, is the condition of positive rank of the Jacobian $J_1^d(16)$ of $X_1^d(16)$.

\subsection{Positive rank of \texorpdfstring{$J_1^d(16)$}{J1d(16)}}

We again start with a preparatory lemma concerning torsion growth of $J_1(16)$ in quadratic fields.

\begin{lemma}
\
\begin{enumerate}
    \item $J_1(16)(\Q)_{tors} \cong \Z/2\Z \oplus \Z/10\Z.$
    \item $J_1(16)(\Q(i))_{tors} \cong \Z/2\Z \oplus \Z/2\Z \oplus \Z/10\Z.$
    \item $J_1(16)(\Q(\sqrt{2}))_{tors} \cong \Z/2\Z \oplus \Z/2\Z \oplus \Z/10\Z.$
    \item $J_1(16)(K)_{tors} \cong \Z/2\Z \oplus \Z/10\Z$ for $K \neq \Q(i), \Q(\sqrt{2})$ any quadratic field.
\end{enumerate}
\end{lemma}

\begin{proof}
    Part (1) may be directly verified in Magma. Parts (2) and (3) follow from a Magma computation that shows that the torsion subgroup is at most $\Z/2\Z \oplus \Z/2\Z \oplus \Z/10\Z$ (\path{TorsionBound}), together with finding the points $(\sqrt{-1},0)$, $(-\sqrt{-1},0)$, $(-1+\sqrt{2},0)$ and $(-1-\sqrt{2},0)$ that are $2$-torsion (since they are Weierstrass points; note that these correspond to cusps on $X_1(16)$). Part (4) follows in a similar way to the proof of \Cref{lem:no_torsion_growth_18}, by computing the abelian group structure of $\widetilde{J_1(16)}(\F_{p^2})$ for several small $p$ to show that over any quadratic field $K$, one has
    \begin{itemize}
        \item $J_1(16)(K)_{tors}$ is at most $\Z/2\Z \oplus \Z/2\Z \oplus \Z/2\Z \oplus \Z/10\Z$;
    \end{itemize}
    one then concludes by considering $2$-torsion rank which can easily be read of from the factorisation of the hyper-elliptic equation.
\end{proof}

\begin{corollary}\label{cor:pos-rank-16}
    For $d \neq -1, 2$, if $X_1^d(16)(\Q)$ admits a point with nonzero $y$-coordinate, then $J_1^d(16)(\Q)$ has positive rank.
\end{corollary}

\begin{proof}
    A point on $X_1^d(16)(\Q)$ with nonzero $y$-coordinate corresponds to a $K$-point of $X_1(16)$ that is not a $\Q$-point, so the same proof as \Cref{prop:pos-rank-13-18} applies.
\end{proof}

Using the twisted winding element computation from before, we compute the squarefree values of $d$ with $|d| < 10{,}000$ for which the analytic rank of $J_1^d(16)$ is positive; this yields $674$ values. We search for rational points with nonzero $y$-coordinate, and find such points on $55$ of the twists, leaving $619$ values to be dealt with. While we are not able to deal with all of these values, we can deal with the majority of them - $581$ to be specific - via a method due to Bruin known as \emph{Elliptic curve Chabauty}, which we use in conjunction with a Two-cover descent.

\subsection{Elliptic curve Chabauty}

The use of elliptic curve quotients in explicitly carrying out Chabauty's method for the computation of rational points goes back to Bruin's paper \cite{bruin2003chabauty}; the method we use is described in \cite[Section 8]{bruin2009two}, which we now briefly summarise. For simplicity here $C$ will denote a hyperelliptic curve of genus $2$ over $\Q$, although the method works for higher genus hyperelliptic curves over arbitrary number fields. It is based on the idea that, even if $C(\Q)$ is nonempty and so the fake $2$-Selmer group $\Sel^{(2)}_{\textup{fake}}(C/\Q)$ is nonempty, it still contains useful information that can be exploited to fully determine $C(\Q)$.

The main theoretical result of \cite{bruin2009two} is a refined version of the Chevalley-Weil theorem, that every rational point on $C$ lifts to a rational point on one of finitely many $2$-covers $D \overset\pi\rightarrow C$; the algorithmic result is that one can explicitly construct these covers $D$. So if, for each $D$, we can determine $\pi(D(\Q))$, then we are done. The problem is that $D$ has large genus, so computing $D(\Q)$ is difficult. The idea is to work with other quotients of $D$ besides $C$. Indeed, if $C$ is given by a model $y^2 = f(x)$ with $f$ of odd degree, then by taking a degree $3$ factor of $f$ over some number field $L$, and by taking an appropriate twist $\gamma_D$, one may construct the elliptic curve defined over $L$: \[ E_D : \gamma_Dy^2 = g(x) \] together with an $L$-rational map $D \to E_D$. The Chabauty condition is that, if $\rk(E_D(L)) < [L : \Q]$, then $x(E_D(L)) \cap \P^1(\Q)$ is finite and explicitly computable by \cite{bruin2003chabauty}. And since $x(\pi(D(\Q))$ is contained in this finite set, then so is $D(\Q)$. If this method successfully manages to compute $x(E_D(L)) \cap \P^1(\Q)$ and prove that it is finite for every $2$-covering $D$ in the fake 2-Selmer set of $C$, then one can successfully determine $C(\Q)$. The algorithm is implemented in Magma as the intrinsic \path{Chabauty} (note that this is overloaded: the same intrinsic works for the classical Chabauty-Coleman method; the type of the parameters passed to it determine which is used).

In our case, the polynomial $f(x)$ that determines the model of $X_1(16)$ we are working with is highly composite, meaning that the number fields $L$ arising in the above construction will always be quite small (of degrees  $1,2$ or $4$), which aids the computation. Our implementation, as well as the execution of it, may be found in \path{computations/x1_16_chabauty.m}. (The point search occurs in \path{computations/x1_16_point_search.m}.)

The values for which this method did not succeed are as follows:
\begin{equation}\label{eqn:x1_16_fail}
  \begin{aligned}
        -&8259, -7973, -7615, -7161, -7006, -6711, -6503, -6095, -6031,\\
        -&6005, -4911, -4847, -4773, -4674, -4371, -4191, -4074, -3503,\\
        -&3199, -1810, -1749, -815, 969, 1186, 3215, 3374, 3946, 4633, 5257,\\
        &5385, 7006, 7210, 7733,8459, 8479, 8569, 9709, 9961. 
\end{aligned}
\end{equation}

In particular, we see that we are able to deal with all values in the range $|d| < 800$, completing the proof of part (4) of \Cref{thm:main}.

\begin{remark}
    It would be interesting to attempt to deal with the above values for which Elliptic curve Chabauty failed using Two-cover descent together with Quadratic Chabauty on curves $D'$ intermediate to $D$ and $C$. We did attempt to run Quadratic Chabauty directly on these twists of $X_1(16)$, but in each case the Picard rank was $1$ (as may be verified with the code associated to \cite{costa2019rigorous}), which is a nonstarter for that method. However, a combination method may be successful, and would be of interest to consider further.
    
    In particular, if one can use this to determine the rational points on the two twists $C=X_1^{-815}(16)$ and $C=X_1^{969}(16)$ of $X_1(16)$, one would have established explicit uniformity of torsion over quadratic fields $\Q(\sqrt{d})$ for all $|d| < 1000$.
\end{remark}

\section{Comparison of our results with a conjecture of Granville}\label{sec:results}

In this section we report on the results of the computation for the genus $2$ curves $X_1(13)$, $X_1(16)$, and $X_1(18)$. We focus only on the genus $2$ cases because the question of when twists of elliptic curves have positive rank has already been well-studied in the literature; see for example \cite{watkins2014ranks} for computational work in this direction, \cite{li2018recent} for an overview of what was known as of 2018, and Smith's recent work \cite{smith2022distribution} showing that Goldfeld's conjecture (that asymptotically $50\%$ of twists in a quadratic twist family have rank $0$ and $50\%$ have rank $1$) follows from BSD under some additional assumptions on the level-$2$ structure of the elliptic curve. (In forthcoming work of Smith \cite{smith2024bsd} these additional assumptions have been removed.)

Our data can be compared to work of Granville \cite{granville2007rational} that studies how many twists of a given hyperelliptic curve admit nontrivial rational points. By nontrivial, Granville means those with $y$-coordinate $0$, as well as the points at $\infty$ when the defining polynomial of the curve has odd degree, so this is exactly the situation we have studied in \Cref{sec:13_18,sec:16}. In particular, Granville makes the following conjecture about the number of twists that admit such a nontrivial rational point.

\begin{conjecture}[Granville, part of Conjecture 1.3 in \cite{granville2007rational}]
    Let $C$ be a hyperelliptic curve over $\Q$ of genus $g \geq 2$, defined by a model $y^2 = f(x)$ for $f \in \Z[x]$ that does not have repeated roots. Then there exists a positive constant $\kappa_f'$ such that there are $\sim\kappa_f'B^{1/(g+1)}$ squarefree integers $d$ with $|d| < B$ for which the quadratic twist $C_d$ has a nontrivial rational point.
\end{conjecture}

\begin{remark}
    Granville makes a similar conjecture about integral points that also applies to elliptic curves; this part of the conjecture is not relevant for our purposes so we omit it here.
\end{remark}

The basis upon which this may be elevated to the stature of a conjecture is one of the main theorems of that paper; namely that this conjecture follows from the $abc$-conjecture provided various assumptions on $f$ are satisfied (see Theorem 1.4 in \emph{loc. cit.}). These conditions do not cover our case of $g = 2$, so in this case, the conjecture is still open even if one assumes the $abc$-conjecture. In this section, we wish to see if our computations agree with the above conjecture of Granville; that is, we will study the growth of $|T_B(N)|$ as $B$ grows, for $N = 13$, $16$ and $18$, and compare it to $\kappa_f'B^{1/3}$.

In Section 1.1 of his paper, Granville gives a formula for the $\kappa_f'$ constant, which we now briefly review in our case of $g = 2$ and with various simplifications; we refer the interested reader to \emph{loc. cit.} for the more general case. To this end, we let $f$ be a monic polynomial of degree $5$ or $6$ with integer coefficients and no repeated roots. We define $F(x,z) := z^6f(x/z)$. For each integer $r$ let $\omega(r)$ be the number of residue classes $t \Mod{r}$ for which $r$ divides $f(t)$, and for $k$ a positive integer write $\omega'(p^k) := p^{k-1}(p - 1)\omega(p^k)$. We define $V_f'$ to be the area of $\left\{ (x,y) \in \R^2 : |F(x,y) \leq 1\right\}$ and $A_f(\Q)$ to be $|\Aut_\Q(C)|/2$, which must equal $1$, $2$, $3$, $4$, $6$, $8$ or $12$. We then have \[ \kappa_f' := \frac{V_f'}{A_f(\Q)}\prod_p\left\{1 + \left(1 - \frac{1}{p^{2/3}}\right)\left(\frac{\omega'(p^2)}{p^{10/3}} + \frac{\omega'(p^4)}{p^{20/3}} + \cdots\right)\right\} \] which converges to a well-defined real number. For $p \nmid \Delta(f)$, the $p$th term of the Euler product is more simply $1 + \omega(p)(p-1)(p^{2/3}-1)/(p^3 - p^{5/3})$.

\begin{remark}
    Granville gives a rather more verbose definition of $A_f(\Q)$ as the number of distinct $\Q$-linear transformations $(x,z) \mapsto (\alpha x + \beta z, \gamma x + \delta z)$ of $F$ for which $F(\alpha x + \beta z, \gamma x + \delta z) \equiv F(x,z) \Mod{(\Q^\ast)^2}$, and $\alpha\delta - \beta\gamma \neq 0$; this is equivalent to how we have written it more succinctly above as $|\Aut_\Q(C)|/2$.  Indeed if $F(\alpha x + \beta z, \gamma x + \delta z) = k^2F(x,y)$ for some rational number $k$ then $y^2-F(x,z) = 0 \Leftrightarrow (ky)^2-F(ax+bz,cx+dz)$. So both $(x,y,z) \mapsto (\alpha x + \beta z,ky,\gamma x + \delta z)$ and $(x,y,z) \mapsto (\alpha x + \beta z,-ky,\gamma x + \delta z)$ are automorphisms of $C$ seen as a curve in weighted projective space. Since every automorphism of $C$ commutes with the hyperelliptic involution one also gets $(x,z) \mapsto (\alpha x + \beta z, \gamma x + \delta z)$ back since every automorphism of $C$ induces an automorphism of $\P^1 = C/h$ where $h$ is the hyperelliptic involution.
\end{remark}

See \path{granville_constants/kappa_consts.py} for the implementation of these constants for the three defining polynomials of interst for us (written explicitly at the beginning of \Cref{sec:13_18,sec:16}). The most challenging aspect of this was the computation of $V_f'$ for $X_1(16)$, which is an unbounded region with $8$ cuspidal spikes; see \path{granville_constants/euclidean_contribution.ipynb} to see this and the other such regions. Note that this implementation does not use interval arithmetic, or yield rigorously proven upper or lower bounds of $V_f'$; this is sufficient for our purpose of getting a sense of the larger picture.

Plotting $|T_B(N)|$ against $B$, and comparing it to $\kappa_f'B^{1/3}$ yields \Cref{fig:graph1}. Here we have assumed that the values we have not been able to decide upon (specifically, $9689$ for $X_1(13)$; $2841$, $4729$ and  $9969$ for $X_1(18)$; and the 38 values in (\ref{eqn:x1_16_fail}) for $X_1(16)$) are not in $T_B(N)$; this seems the most likely outcome for the vast majority of the unhandled cases given that we have searched for points whose $x$-coordinate has na\"ive height at most $10,000$ on each of the relevant twists. Indeed, as can be seen in \path{computations/x1_16_point_search_log.txt}, the vast majority of the rational points found on $X_1^d(16)$ have $x$-coordinate whose na\"ive height is $<100$. Furthermore, the point $$\left(\frac{1681}{882},\frac{479110914870}{882^3}\right)$$ on $X_1^{8570}(16)$ was the only point we found where the height of its $x$ coordinate exceeded $1000$. This is also what is expected according to \cite[Thm. 1.1]{granville2007rational}, which states that under the $abc$-conjecture rational points on twists should have a small $x$-coordinate. 
In any case, these exceptions make up less than 1\%  of the total number of squarefree integers $d$ with $|d| < 10{,}000$ that we can deal with for each of our three curves, so this also does not affect the larger picture.

\begin{figure}
    \centering
    \includegraphics[width=\linewidth]{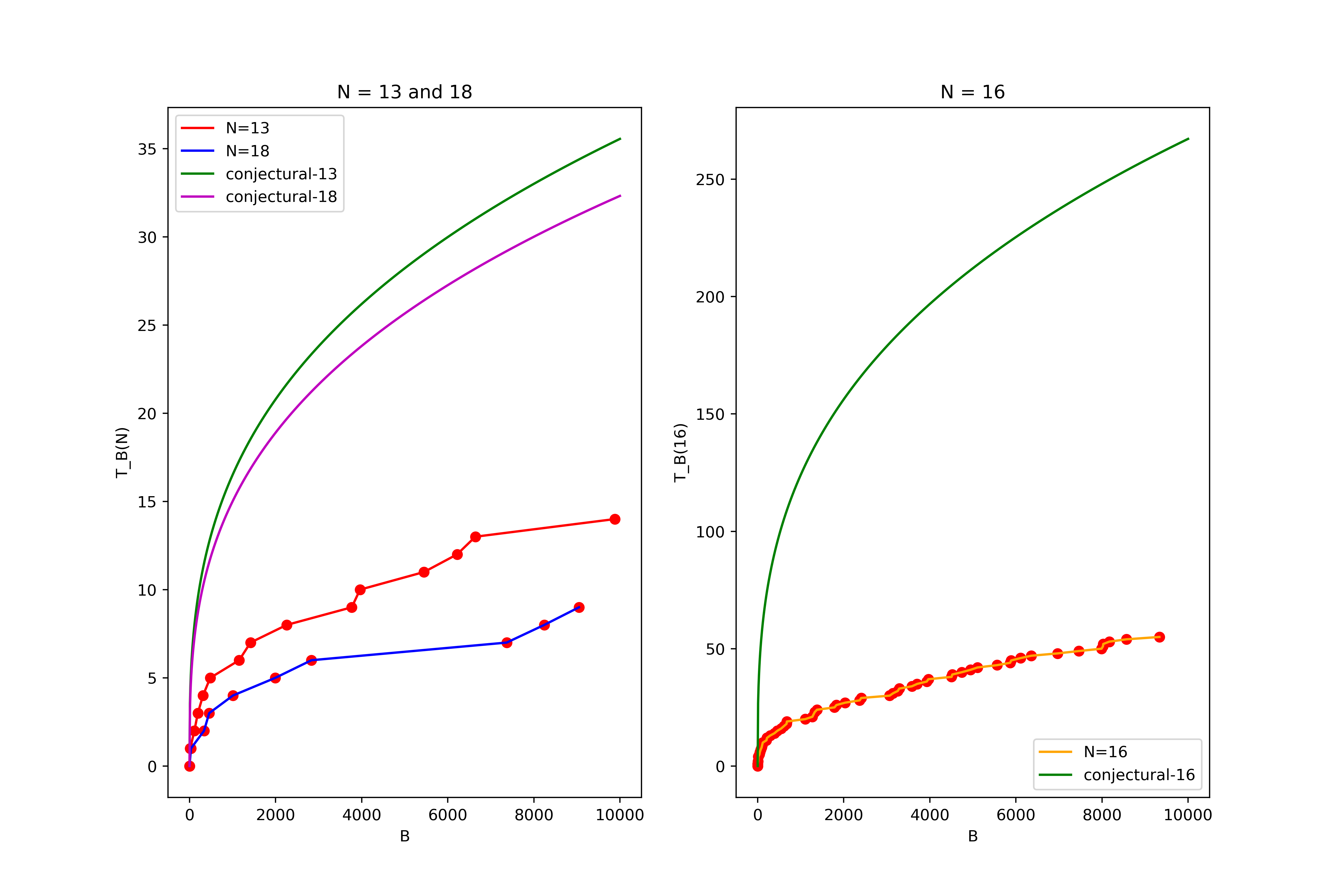}
    \caption{Graphs showing how $|T_B(N)|$ grows with $B$ for $N = 13$, $16$ and $18$, together with the conjectural growth of $\kappa_f'B^{1/3}$; the values of $\kappa_f'$ here are, respectively, $1.65$, $12.4$ and $1.5$.}
    \label{fig:graph1}
\end{figure}

One clearly sees that the conjectural distribution of $\kappa_f'B^{1/3}$ is significantly larger than what the data is suggesting. From \Cref{fig:graph1} even the shape of the asymptotic behaviour is not apparent; so in \Cref{fig:graph2} we have artifically reduced the $\kappa_f'$ constant to show more clearly that the growth of the data is indeed asymptotically proportional to $B^{1/3}$.

\begin{figure}
    \centering
    \includegraphics[width=\linewidth]{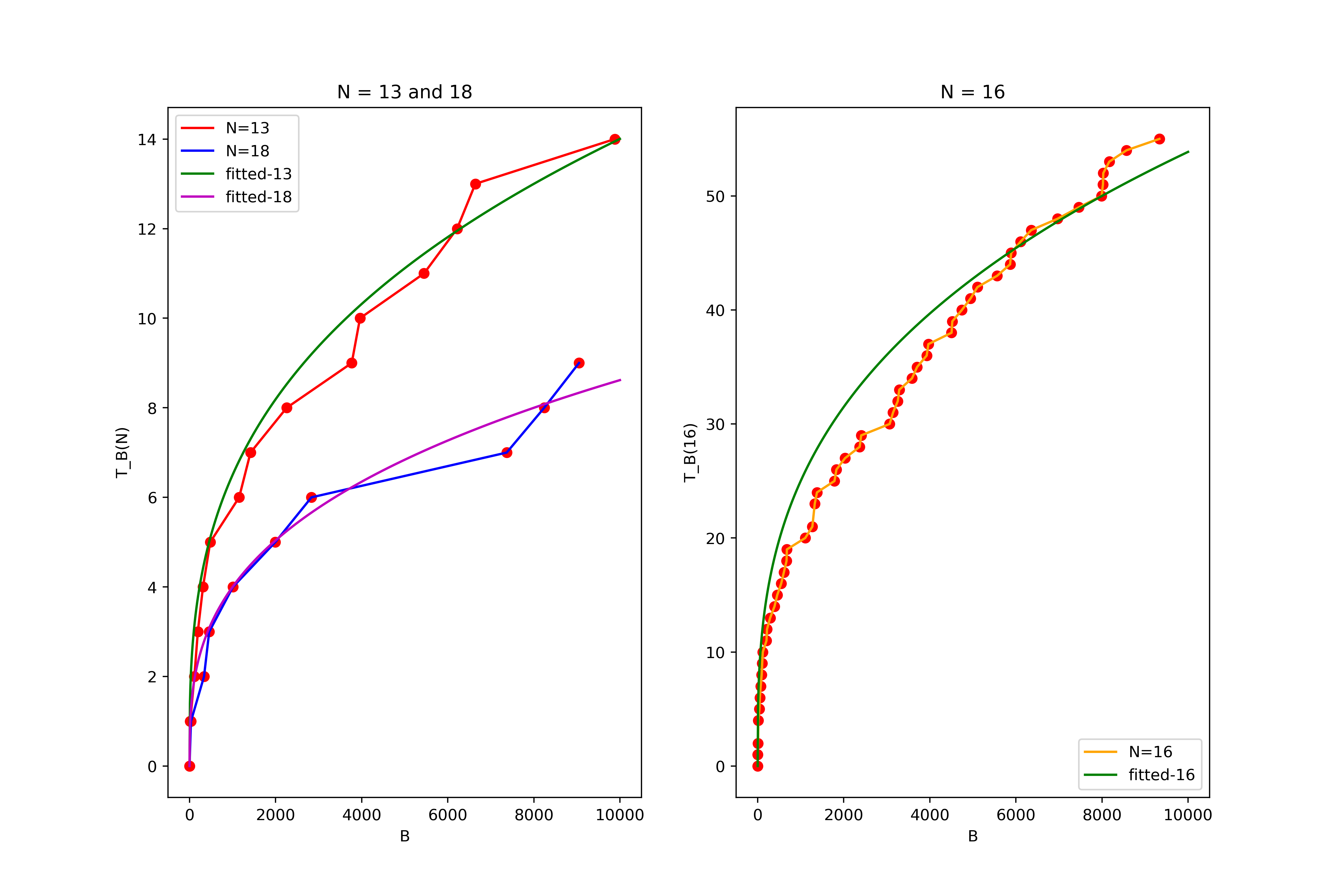}
    \caption{The same graph as \Cref{fig:graph1} but with smaller values of $\kappa_f'$, viz. respectively, $0.65$, $2.5$, $0.4$.}
    \label{fig:graph2}
\end{figure}

Therefore, for $X_1(16)$ (respectively, $X_1(13)$, $X_1(18)$), it seems that the constant $\kappa_f'$ is about 5 (respectively, $2.5$, $3.75$) times too big. To investivate this discrepancy, we plot in \Cref{fig:graph3} $B$ against $|T_B(N)|/B^{1/3}$, which conjecturally should converge to $\kappa_f'$.

\begin{figure}
    \centering
    \includegraphics[width=\linewidth]{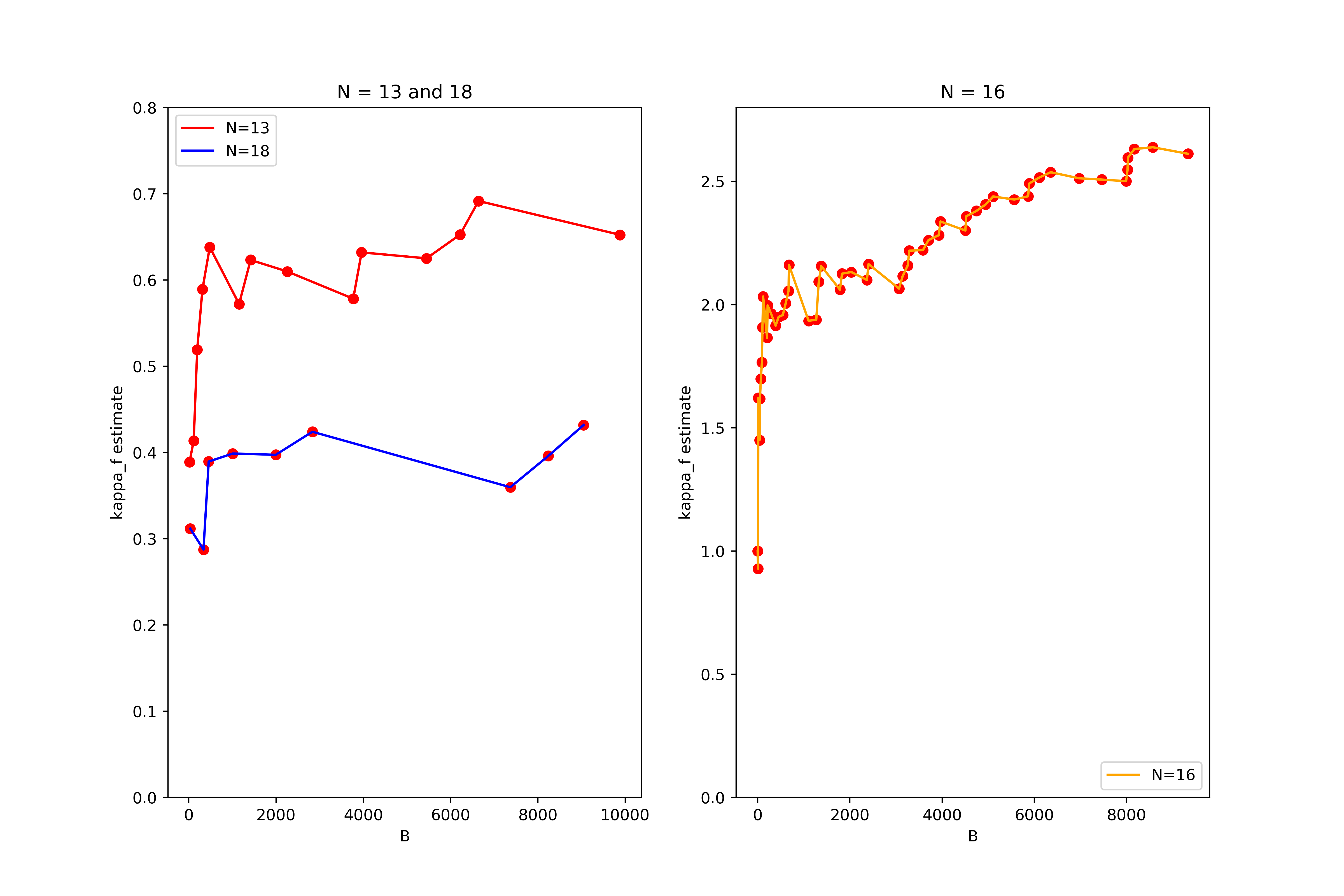}
    \caption{Plot of $B$ against $|T_B(N)|/B^{1/3}$; this should conjecturally converge to $\kappa_f'$.}
    \label{fig:graph3}
\end{figure}

On the graph for $X_1(16)$, there is a clear upward trend, so while it seems to hover at about $2.5$ (the value of the fitted graph), it is not inconceivable that it will continue to drift upwards and reach Granville's value of $12.5$. Put in other words, \Cref{fig:graph3} suggests that our results are not necessarily incompatible with Granville's $\kappa_f'$ constant; more data is needed to determine this.

Other potential reasons for this discrepancy are as follows:

\begin{enumerate}
    \item When computing $\kappa_f'$ we clearly had to take only finitely many summands in the sum for bad primes, and only finitely many primes in the Euler product. However, taking more would only \emph{increase} $\kappa'_f$, making the discrepancy larger.
    \item The value of $V_f'$ was approximate, and involved numerical integration in the real plane. However, it seems unlikely that this would be off by an order of magnitude.
    \item Our assumption that the values we were not able to determine are not actually in $T_B(N)$. However, as explained above, these account for a tiny percentage of the total values (especially so in the case of $X_1(13)$ and $X_1(18)$), and so also would not explain this larger observed discrepancy. Taking the other extreme - that all of these unknown values actually are in $T_B(N)$ - would increase $|T_{10,000}(16)|$ (respectively, $|T_{10,000}(13)|$, $|T_{10,000}(18)|$) by $69\%$ (respectively, $7\%$, $33\%$). This does seem more significant (even if extremely unlikely), but still is not enough to explain the e.g. factor of $5$ discrepancy in the $X_1(16)$ case (much less for the other two curves).
\end{enumerate}

The graphs in this section may be generated with the script \path{graphs/results.py} in the top level.

\section{Future research}\label{sec:future}

In light of the results and discussion of \Cref{sec:results} it would therefore be interesting to obtain more data to fully ascertain the situation regarding the growth of $|T_B(N)|$ as $B$ increases, or for this to be investigated further. More rigorous computation of Granville's $\kappa_f'$ constants would also be of value to undertake.

One of Granville's main theorems in \emph{loc. cit.} is as follows.

\begin{theorem}[Granville, Theorem 1.1(ii) in \cite{granville2007rational}]
Assume that the $abc$-conjecture is true. Suppose that $f(x) \in \Z[x]$ does not have repeated roots and $d \in \Z$. Consider the curve $C_d$ given by $dy^2=f(x)$.
If $g(C_d) \geq 2$ then the rational points on $C_d$ with $x$-coordinate $r/s$ where $(r, s) = 1$ satisfy $$|r|, |s| \ll |d|^{1/(2g(C_d)-2)+o(1)}.$$
\end{theorem} 

The proof of this theorem contains a way of constructing $abc$-triples from points on $C_d$. However this construction relies on a Belyi map on $C$ that factors via the hyperelliptic involution. Now for the modular curves $X_1(N)$ one has that the $j$-invariant $j: X_1(N) \to X(1) \cong \P^1$ only ramifies at $0, 1728, \infty$. So in particular $j/1728$ is a Belyi map. Furthermore the hyperelliptic involutions of $X_1(13)$ and $X_1(16)$ are the diamond operators $\langle 5 \rangle$ and $\langle 7 \rangle$ and hence the $j$-invariant factors via the hyperelliptic map. In particular it should be possible to make Granville's method completely explicit for $X_1(13)$ and $X_1(16)$ by constructing $abc$-triples related to the $j$-invariants of points on $X_1^d(13)$ and $X_1^d(16)$ with exceptionally large $x$ coordinate. It would be of interest for this idea to be fleshed out.

For $X_1(18)$ the hyperelliptic involution is $w_2\langle 7 \rangle$, meaning the $j$-invariant doesn't factor via the hyperelliptic map and hence the idea sketched above will not work here. It would thus be interesting to find another idea that would deal with this case.

In a different direction, one could consider the uniformity of torsion question for cubic fields, given that we now have the list of groups in this case. Note that Bruin and Najman \cite{bruin2016criterion} have found all possible torsion subgroups over the cyclic cubic field $\Q(\zeta_{13} + \zeta_{13}^5 + \zeta_{13}^8 + \zeta_{13}^{12})$ as well as the quartic field $\Q(\zeta_5)$, so this would be a good place to start with this investigation.

\bibliographystyle{alpha}
\bibliography{references.bib}{}
\end{document}